\documentclass[final, onefignum, onetabnum]{siamart171218}
\usepackage{amssymb,overpic} \usepackage{amsmath}
\usepackage{hyperref}
\usepackage{mathdots}
\usepackage{enumitem}
\usepackage{cite}
\usepackage{tikz}
\usepackage{todonotes}

\usetikzlibrary{shapes.misc}

\graphicspath{{Images/}}

\newcommand{\RR}{\mathbb R}

\newcommand{\OO}{\mathcal O}

\newcommand{\parens}[1]{\left( #1 \right)}

\newcommand{\partialf}[2]{\frac{\partial #1}{\partial #2}}

\begin{document}
\date{\today}
\title{A spectral element method for meshes with skinny elements}
\author{Aaron Yeiser\\ MIT, Cambridge, MA 02139, US (\MakeLowercase{\texttt{ayeiser@mit.edu}}) \\[.2cm] Advisor: Alex Townsend \\ Department of Mathematics, Cornell University, Ithaca, NY 14853.}
\maketitle

\begin{abstract}
  When numerically solving partial differential equations (PDEs), the first step is often to discretize the geometry using a mesh and to solve a corresponding discretization of the PDE.
  Standard finite and spectral element methods require that the underlying mesh has no skinny elements for numerical stability.
  Here, we develop a novel spectral element method that is numerically stable on meshes that contain skinny elements, while also allowing for high degree polynomials on each element.
  Our method is particularly useful for PDEs for which anisotropic mesh elements are beneficial and we demonstrate it with a Navier--Stokes simulation.
  Code for our method can be found at \url{https://github.com/ay2718/spectral-pde-solver}.
\end{abstract}



\section{Introduction}\label{sec:Intro}
A central idea for numerically solving partial differential equations (PDEs) on complex geometries is to represent the domain by a mesh, i.e., a collection of nonoverlapping triangles, quadrilaterals, or polygons. The mesh is then used by a spectral element method to represent the solution on each element by a high degree polynomial~\cite{Patera_84_01,Solin_03_01}. Engineering experience suggests that meshes containing skinny or skewed elements are undesirable and should be avoided, if possible~\cite{Shewchuk_02_01}. However, from a mathematical point-of-view, long and thin elements can better resolve anisotropic solutions~\cite{Shewchuk_02_01} that appear in computational fluid dynamics and problems involving boundary layers.

In this paper, we develop a spectral element method that is numerically stable on meshes containing elements with large aspect ratios and also allows for high degree polynomials on each element. In particular, the numerical stability of our method is observed to be independent of the properties of the mesh and polynomial degrees used. This has several computational advantages:  (1) A practical $hp$-refinement strategy does not need to take into consideration the quality of the final mesh, (2) Anisotropic solutions can be represented with fewer degrees of freedom, and (3) Low-quality meshes can be used without concern for numerical stability.

Over the last two decades, there has been considerable research on mesh generation algorithms that avoid unnecessarily skinny elements~\cite{Ruppert_95_01,Shewchuk_02_02}. While these mesh generation algorithms are robust and effective, they can be the dominating computational cost of numerically solving a PDE. By allowing low-quality meshes, our method alleviates this current computational burden on mesh generation algorithms. This is especially beneficial for time-dependent PDE in which an adaptive mesh is desired, for example, to resolve the propagation of a moving wave front.

There are a handful of mesh generation algorithms that have been recently designed for anisotropic solutions that generate meshes with skinny elements~\cite{Schoen_08_01}. It is therefore timely to have a general-purpose element method that allows for an accurate solution on such meshes. For time-dependent
PDEs, such as the Navier--Stokes equations, skinny elements can relax the Courant--Friedrichs--Lewy time-step restrictions, leading to significant computational savings in long-time incompressible fluid simulations.

A standard element method constructs a global discretization matrix with a potentially large condition number when given a mesh containing skinny elements. This is because the global discretization matrix is assembled element-by-element from a local discretization of the PDE on each element. Since the elements may have disparate areas, the
resulting discretization matrix can be poorly scaled. One can improve the condition number of the global discretization matrix by scaling the local discretization matrices by the element's area.  Since we are using a spectral element method, we can go a little further and scale the local discretization matrices by the Jacobian of a bilinear transformation that maps the element to a reference domain (see Section~\ref{sec:Ultraspherical}).

The solutions of uniformly elliptic PDEs with Dirichlet boundary conditions defined on a convex skinny element are not sensitive to perturbations in the Dirichlet data or forcing term.
Since we know it is theoretically possible to solve elliptic PDEs on meshes with skinny elements, it is particularly disappointing to employ a numerically unstable method (see Section~\ref{sec:CondNumber}).

The structure of the paper is as follows:
In Section~\ref{sec:CondNumber}, we define some useful terminology and provide the theory to show that certain elliptic PDEs on skinny elements are well-conditioned.
Section~\ref{sec:Ultraspherical} describes the process of using our spectral element method on a single square element, and Section~\ref{sec:Determinant} extends that process to convex quadrilateral domains.
In Section~\ref{sec:Interface}, we define a method for joining together two quadrilateral domains along an edge, and Section~\ref{sec:ElementMethod} extends that method to joining an entire quadrilateral mesh.
Key optimizations to our method are described in Section~\ref{sec:Optimization}.
Finally we present a Navier--Stokes solver using our spectral element method in Section~\ref{sec:NavierStokes}.

\section{Uniformly elliptic PDEs on skinny elements}\label{sec:CondNumber}
It is helpful to review some background material on skinniness of elements and the condition number of PDEs on skinny elements.
Given an element $\mathcal{T}$ of any convex 2D shape, let $r_{\text{in}}$ denote the {\em inradius}, i.e., the radius of largest circle that can be inscribed in $\mathcal{T}$, and
$r_{\text{out}}$ denote the {\em min-containment radius}, i.e.,  the radius of smallest circle that contains $\mathcal{T}$. We define the {\em skinniness} of $\mathcal{T}$
to be
\[
s\left( \mathcal{T}\right) =  \frac{r_{\text{in}}}{r_{\text{out}}}.
\]
For finite element methods, it is widely considered that a skinny triangle with a large angle, i.e., $\theta \approx \pi$, is more detrimental to numerical stability than those triangles with only acute angles. This is based on the theoretical results on interpolation and discretization errors in~\cite{Babuska_76_01}.  The spectral element method that we develop is numerically stable on meshes containing any type of skinny quadrilateral.  Therefore, it is also numerically stable on meshes containing any kind of skinny triangle, because a triangle can be constructed out of three non-degenerate quadrilaterals by dissecting along the medians, as in Figure \ref{fig:ManyQuadError}.

\subsection{Apriori bounds}

Consider the second-order elliptic differential operator operating on domain $\Omega$.
\[
\mathcal{L} u = \sum_{i,j=1}^d a_{ij}(x) \frac{\partial^2 u}{\partial x_i\partial x_j} + \sum_{i=1}^d b_i(x)\frac{\partial u}{\partial x_i}+ c(x) u,
\]
where the coefficients $a_{ij}$, $b_i$, and $c$ are continuous functions on $\Omega$. We say that $\mathcal{L}$ is uniformly elliptic if for some constant $\theta_0 > 0$,
\[
\sum_{i,j=1}^d a_{ij}(x) \xi_i\xi_j \geq \theta_0 \sum_{i=1}^d\xi_i^2,  \qquad x\in\Omega, \qquad \xi \in \RR^d.
\]

\begin{lemma}
\label{poisson}
Consider the uniformly elliptic partial differential operator
$$\mathcal{L} u = \sum_{i, j = 1}^n a_{ij} u_{x_i x_j} + \sum_{i = 1}^n b_i u_{x_i},$$
with coefficients $a_{ij}$ and $b_i$ continuous.  Let $\mathcal{L} u = f$ on domain $\Omega$ with boundary $\partial \Omega$.
If $f \geq 0$ on $\Omega$, then the maximum value of $u$ occurs on $\partial \Omega$.
If $f \leq 0$ on $\Omega$, then the minimum value of $u$ occurs on $\partial \Omega$.
\end{lemma}

\begin{proof}
For a proof, see Theorem 1 in Section 6.5.1 of \cite{evans}.
\end{proof}

\begin{lemma}
\label{screened}
Consider the uniformly elliptic partial differential operator
$$\mathcal{L} u = \sum_{i, j = 1}^n a_{ij} u_{x_i x_j} + \sum_{i = 1}^n b_i u_{x_i} + cu,$$
with coefficients $a_{ij}$, $b_i$, and $c$ continuous.  Let $\mathcal{L} u = f$ on domain $\Omega$ with boundary $\partial \Omega$, and let $c < 0$ on $\Omega$.
If $f \geq 0$ on $\Omega$, then the maximum positive value of $u$ occurs on $\partial \Omega$.
If $f \leq 0$ on $\Omega$, then the minimum negative value of $u$ occurs on $\partial \Omega$.
\end{lemma}

\begin{proof}
For a proof, see Theorem 2 in Section 6.5.1 of \cite{evans}.
\end{proof}

Theorem \ref{boundary} uses Lemmas \ref{poisson} and \ref{screened} to show that perturbations to the boundary conditions of a uniformly elliptic PDE with $c \leq 0$ will not be amplified in solution on the interior of the domain.

\begin{theorem}
\label{boundary}
Let $\mathcal{L} u = \sum_{i, j = 1}^n a_{ij} u_{x_i x_j} + \sum_{i = 1}^n b_i u_{x_i} + cu$, with $c \leq 0$.  Let $u$ be the solution to $\mathcal{L} u = f$ on domain $\Omega$ with boundary $\partial \Omega$, and $u(\partial \Omega) = g$.  If $v$ is the solution to $\mathcal{L} u = f$ with boundary conditions $u(\partial \Omega) = g + \epsilon$, then $\max |v - u| \leq \max |\epsilon|$.
\end{theorem}

\begin{proof}
Let us write $v = u + \delta$.  Since $L$ is a linear operator, we know that $\mathcal{L} v = \mathcal{L} u + \mathcal{L} u \delta = f$.
Therefore, $\mathcal{L} \delta = 0$, with $\delta(\partial \Omega) = \epsilon$.  By Lemma~\ref{screened}, we know that if $\delta$ achieves a nonnegative maximum on the interior of $\Omega$, then $\delta$ is constant.  Similarly, if $\delta$ achieves a nonpositive minimum on the interior of $\Omega$, then $\delta$ is constant.  If $\max |\delta| > \max |\epsilon|$, then $\delta$ will achieve a nonpositive minimum or a nonnegative maximum on the interior of $\Omega$.  However, Lemma~\ref{screened} states that $\delta$ must then be constant, so therefore $\max |\delta| = \max |\epsilon|$, hence a contradiction.
Therefore, $\max |\delta| \leq \max |\epsilon|$.
\end{proof}

Theorem \ref{weakrighthand} provides a weak bound for the effect of perturbations to the right hand side on uniformly elliptic PDEs with $c \leq 0$.

\begin{theorem}
\label{weakrighthand}
Let $\Omega\in\mathbb{R}^d$ be a bounded domain and $\mathcal{L}$ a second-order uniformly elliptic differential operator with $c\leq 0$. Let $u$ be a weak solution of
\[
{\left\lbrace \begin{aligned}
\mathcal{L} u &= f \ \text{ on } \ \Omega,\\
u &= g \  \text{ on } \ \partial\Omega,
\end{aligned} \right. }
\]
and $\hat{u}$ be a weak solution of
\[
{\left\lbrace \begin{aligned}
\mathcal{L} \hat{u} &= f + \Delta f \ \text{ on } \ \Omega,\\
\hat{u} &= g + \Delta g \ \text{ on } \ \partial\Omega,
\end{aligned} \right.}
\]
where $u,\hat{u}\in \mathcal{C}^2(\Omega)\cap \mathcal{C}^0(\partial\Omega)$ and $\partial\Omega$ is the closure of $\Omega$. Then,
\[
\sup_{x\in\Omega} |u(x)-\hat{u}(x)| \leq \sup_{x\in\partial\Omega} |\Delta g(x)| + \left(e^{(\beta+1){\rm diag}(\Omega)}-1\right)\sup_{x\in\Omega} \left|\frac{\Delta f(x)}{\theta(x)}\right|,
\]
where $\beta = \max_{i=1}^d \sup_{x\in\Omega} |b_i(x)/\theta(x)|$.
\end{theorem}
\begin{proof}
Since $u$ satisfies $\mathcal Lu = f$ on $\Omega$ and $u = g$ on $\partial\Omega$, $\Delta u = \hat{u} - u$ satisfies $\mathcal{L}(\Delta u) = \Delta f$ on $\Omega$ and $\Delta u = \Delta g$ on $\partial\Omega$.
One can apply the bound in Theorem 3.7 of~\cite{Gilbarg_15_01}. 
\end{proof}

Theorem \ref{boundary} shows that solutions to uniformly elliptic PDEs with $c\leq 0$ are not sensitive to perturbations in the Dirichlet boundary data and Theorem~\ref{weakrighthand} shows the sensitivity of the solution to perturbations in the right--hand side does not depend on the aspect ratio of the domain $\Omega$.
However, this bound is extremely weak.  We provide a better bound for the screened Poisson equation with Theorem~\ref{righthand}.

\begin{theorem}
\label{righthand}
Let $\mathcal{L} u = u_{xx} + u_{yy} - k^2 u$, and let $u$ be the solution to $\mathcal{L} u = f$ on domain $\Omega$ with boundary $\partial \Omega$, and $u|_{\partial \Omega} = g$.  Let $r$ be the radius of the smallest circle completely containing $\Omega$.  If $s$ is the solution to $\mathcal{L} s = f + \epsilon$ with boundary conditions $s|_{\partial \Omega} = g$, then $\max |s - u| \leq \frac{\max |\epsilon| r^2}{4}$.
\end{theorem}
\begin{proof}
First, consider the smallest disk $\omega$ with radius $r$ that completely contains region $\Omega$.  Let this disk have center $(a, b)$.
For $(x, y) \in \omega$ and $c < 0$, the solution $t = c\parens{(x - a)^2 + (y - b)^2 - r^2}$ is nonnegative, $\max t = -r^2 c$, and $t_{xx} + t_{yy} = 4c$.

Next, consider solution $w$ such that $w_{xx} + w_{yy} = 4c$ for $c < 0$ and $w = 0$ on $\partial \Omega$.
The solution $w - t$ satisfies $(w - t)_{xx} + (w - t)_{yy} = 0$ and $w - t \leq 0$ on $\partial \Omega$, since $t \geq 0$ for all $(x, y) \in \Omega$.
By Lemma~\ref{poisson}, $(w - t) \leq 0$ for all $(x, y) \in \Omega$.
Therefore, $\max w \leq -r^2 c$.

Finally, we will consider the solution $v$ such that $v_{xx} + v_{yy} = \epsilon$ with $v = 0$ on $\partial \Omega$ and $\max |\epsilon| = c$.  By Lemma~\ref{poisson}, $v$ is bounded above by $w_+$ satisfying $\nabla^2 w_+ = -c$ and $w_+ = 0$ on $\partial \Omega$, and $v$ is bounded below by $w_-$ satisfying $\nabla^2 w_- = c$ and $w_- = 0$ on $\partial \Omega$.
We can write $\epsilon = \epsilon_+ + \epsilon_-$, where $\epsilon_+ = \max(\epsilon, 0)$, and $\epsilon_- = \min(\epsilon, 0)$.

Let $\delta_+$ be the solution to $\mathcal{L} \delta_+ = \epsilon_-$ with zero Dirichlet boundary conditions and let $\delta_-$ be the solution to $\mathcal{L} \delta_- = \epsilon_+$ with the same boundary conditions.  By Lemma~\ref{screened}, $\delta_+ \geq 0$ on $\Omega$, and $\delta_- \leq 0$.

We can write that $\nabla^2 \delta_+ = k^2 \delta_+ + \epsilon_-$.  Since $\delta_+ \geq 0$, $\min(k^2 \delta_+ + \epsilon_-) \geq \min(\epsilon_-)$, so $\delta_+$ is bounded above by the solution of $\nabla^2 v = \epsilon_-$ with zero Dirichlet boundary conditions by Lemma~\ref{poisson}.  Similarly, $\delta_-$ is bounded below by the solution of $\nabla^2 v = \epsilon_+$ with zero Dirichlet boundary conditions.  Therefore, $s = u + \delta_+ + \delta_-$.
Thus, $\max |s - u| = \max(\max \delta_+, \ \max -\delta_-) \leq \frac{\max|\epsilon|r^2}{4}$.
\end{proof}

Theorem \ref{boundary} proves that any change to the boundary conditions of the screened Poisson equation will not be amplified in the solution on any domain.
Theorem \ref{weakrighthand} gives a weak bound for solving all second--order uniformly elliptic PDEs with $c \leq 0$, and Theorem \ref{righthand} shows that the amplitude of a perturbation in the right hand side of the screened Poisson equation is bounded by a function of the radius of the domain.  Therefore, solving the screened Poisson equation on any domain is a well-posed problem.

It seems unfortunate to have an unstable numerical algorithm for a well-conditioned problem. Therefore, we ask ourselves the question: Can we create a numerically stable algorithm for solving PDEs on meshes with skinny elements?
\subsection{Stability and conditioning}

Our method represents any linear differential operator as a matrix operator.  One method of assessing the numerical stability of the method is to examine the condition number of the matrix.
In order to demonstrate that the condition number is bounded, let us define a skinny triangle with coordinates $(0, 0), (1, 1 + \epsilon), (2, 2 - \epsilon)$.  We can now define a skinny quadrilateral with vertices $(0, 0), (0.5, 0.5 + 0.5\epsilon), (1, 1), (1, 1 - 0.5\epsilon)$.
When we construct a differential operator matrix on that skinny quadrilateral, we still need a preconditioner.  In fact, we found that row scaling so that each row has a supremum norm of 1 is fairly close to optimal.
Figure \ref{condition number table} shows the condition number $\kappa$ of the normalized matrix for solving Poisson's equation as $\epsilon$ approaches zero.
Even as the quadrilateral becomes very skinny, the condition number of the matrix is bounded from above.
The process of creating the matrix is shown in sections~\ref{sec:Ultraspherical}~and~\ref{sec:Determinant}.

\begin{table}
\centering
\begin{tabular}{c | c | c | c | c | c | c | c}
$\epsilon$ & $1$ & $10^{-1}$ & $10^{-2}$ & $10^{-3}$ & $10^{-6}$ & $10^{-9}$ & $10^{-12}$\\ \hline
$\kappa \ (\times 10^4)$ & 0.1597 & 0.3161 & 0.9632 & 1.0763 & 1.0832 & 1.0832 & 1.0832
\end{tabular}
\label{condition number table}
\caption{The condition number of our matrix for solving Poisson's equation on a skinny quadrilateral with vertices $(0, 0), (0.5, 0.5 + 0.5\epsilon), (1, 1), (1, 1 - 0.5\epsilon)$.  As $\epsilon$ approaches 0, the condition number is bounded from above.}
\end{table}

Figures~\ref{fig:OneQuadError} and~\ref{fig:ManyQuadError} show that our method gives extremely accurate results even on meshes with skinny triangles.  Figure~\ref{fig:PoissonSkinnyQuad} shows that our method performs equally well on meshes with or without skinny elements.

\begin{figure}
\label{fig:OneQuadError}
\centering
\includegraphics[width = 0.75\textwidth]{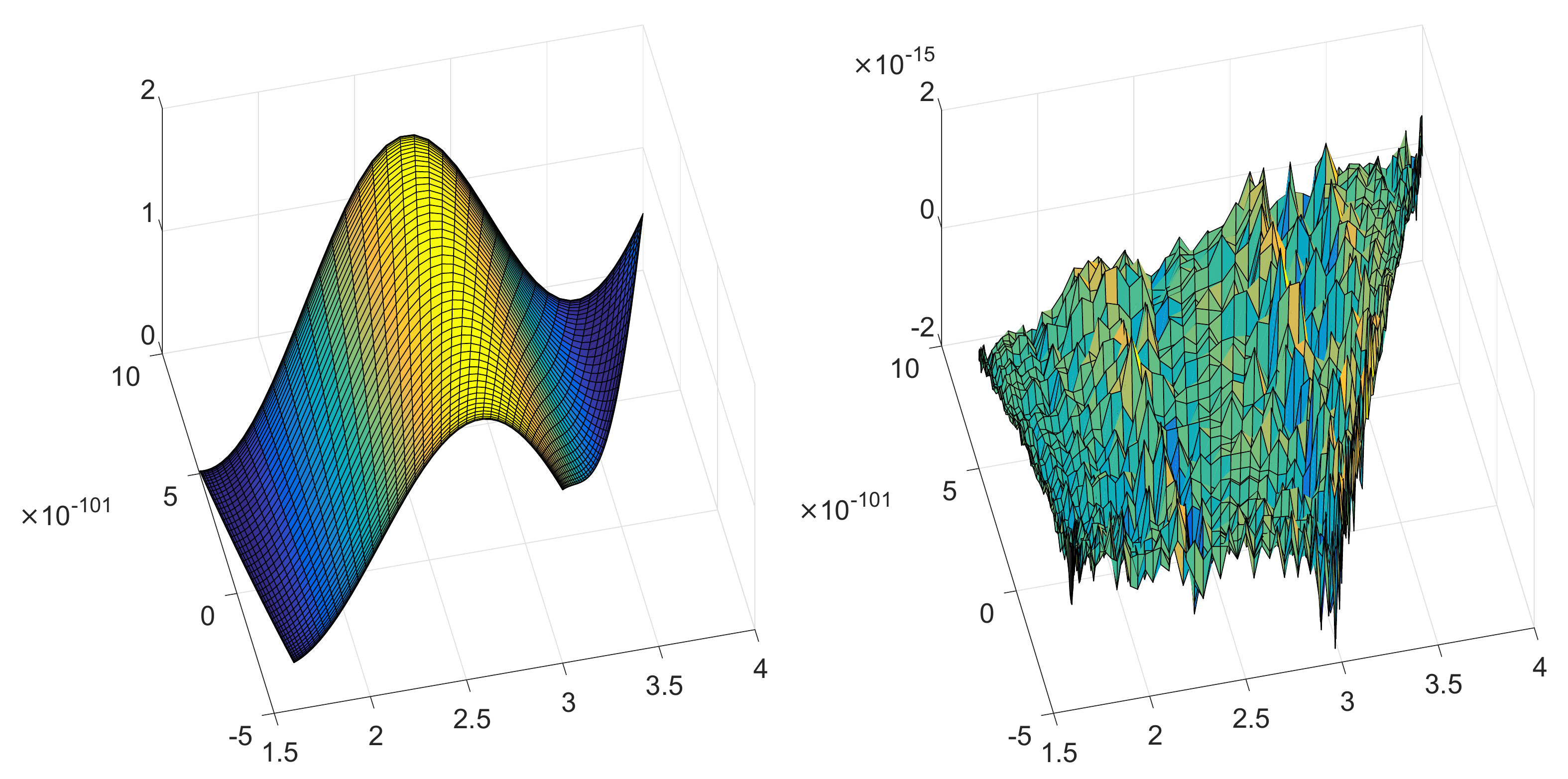}
\caption{The left graph shows a solution to Poisson's equation on quadrilateral $10^{100}$ times as long as it is wide.  The right graph shows the error between the computed solution and the exact solution.  Even on skinny quadrilaterals, our method is still numerically stable.
Our only limitation on the skinniness of mesh elements is roundoff error in actually defining the element.}
\end{figure}

\begin{figure}
    \centering
    \begin{tabular}{p{0.5 \textwidth} p{0.5\textwidth}}
        \includegraphics[width = 0.48 \textwidth]{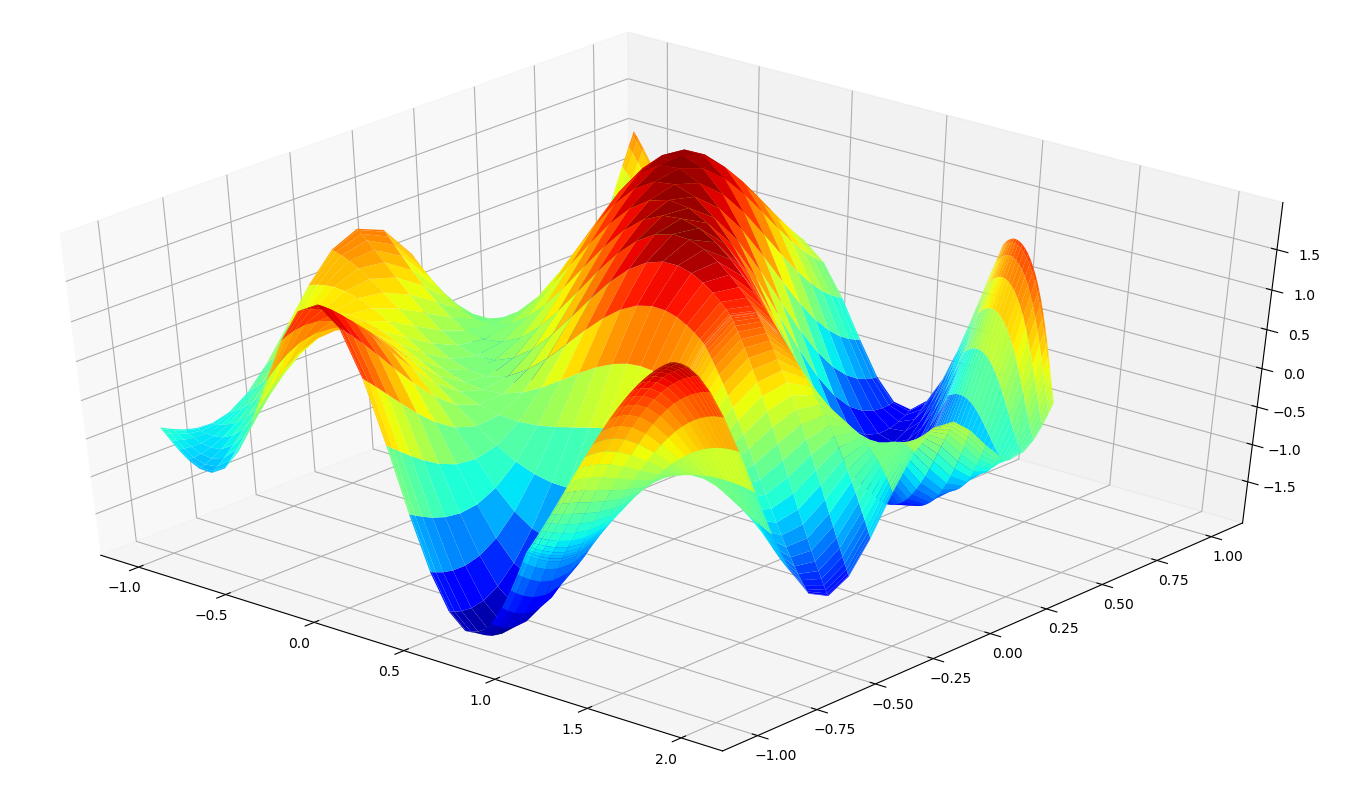} & \includegraphics[width = 0.48 \textwidth]{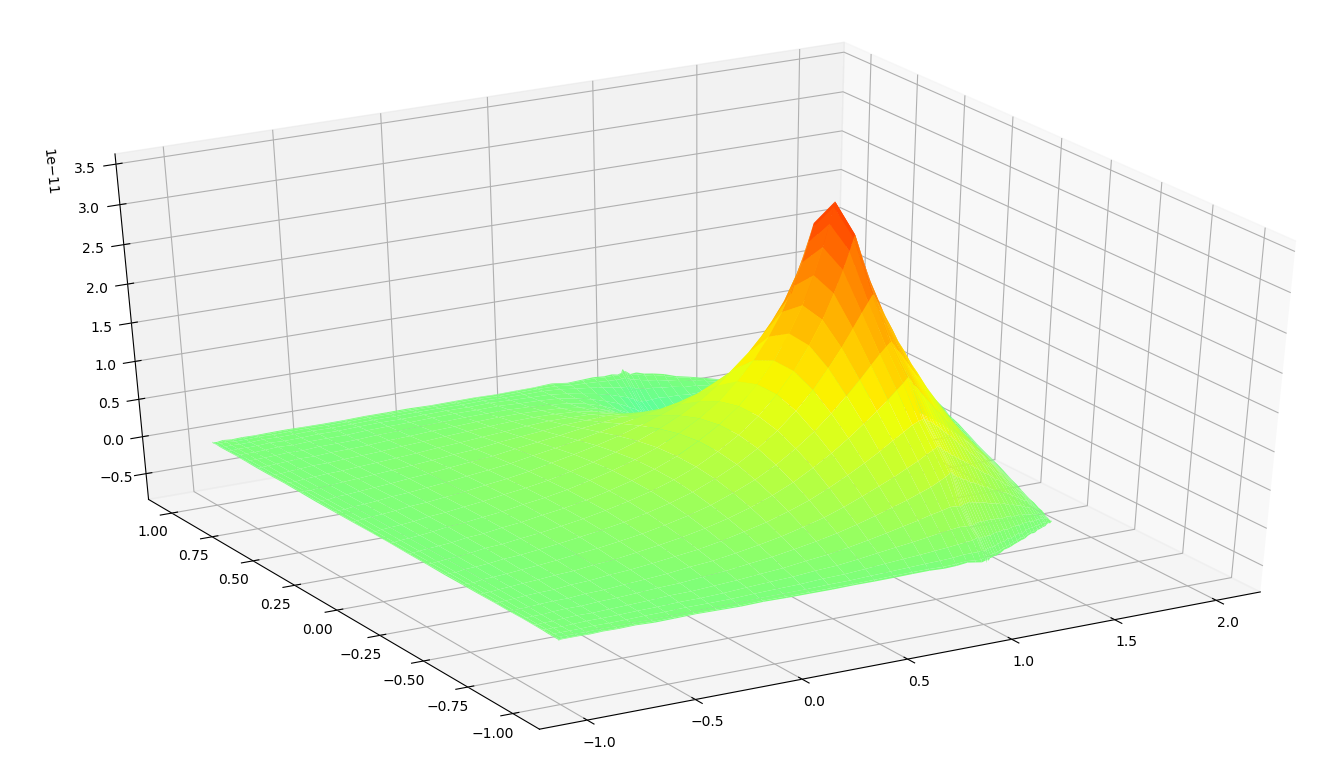} \\
        {\em A solution to Poisson's equation on two elements with low aspect ratio} & {\em Error between the exact solution and the computed solution $(< 10^{-10})$} \\
        \includegraphics[width = 0.48 \textwidth]{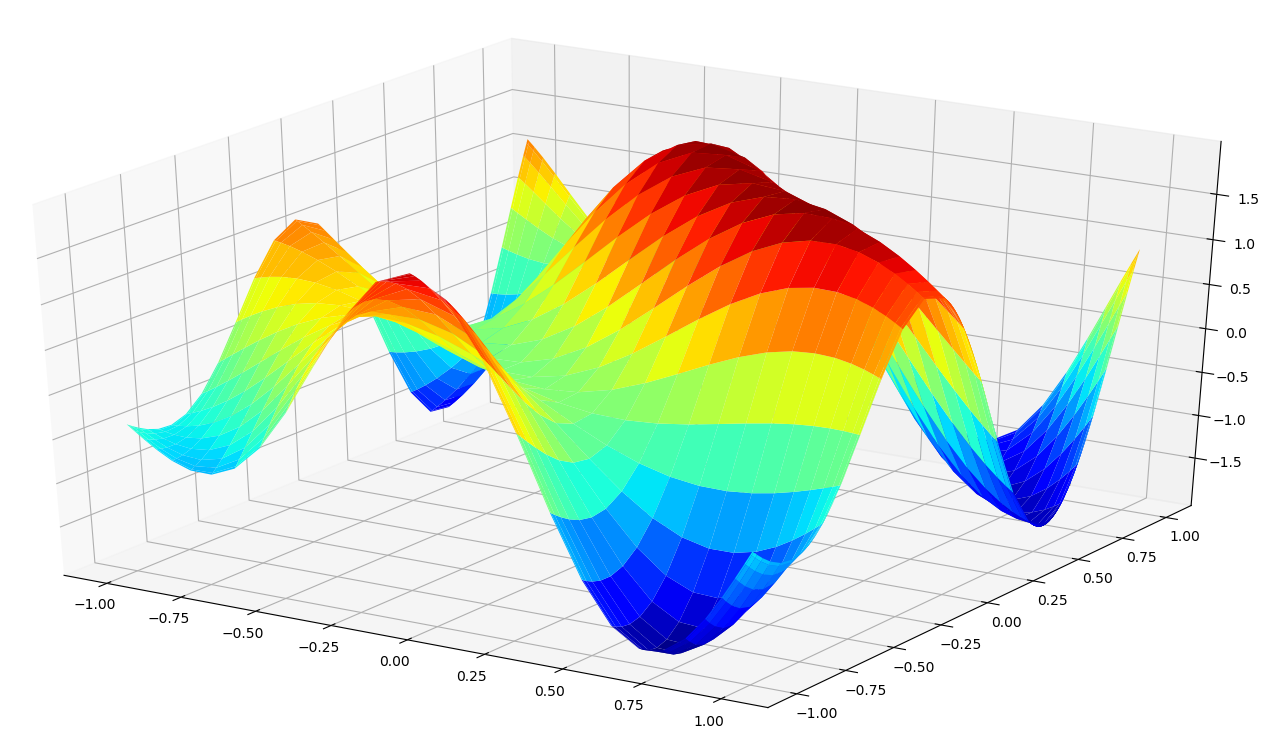} & \includegraphics[width = 0.48 \textwidth]{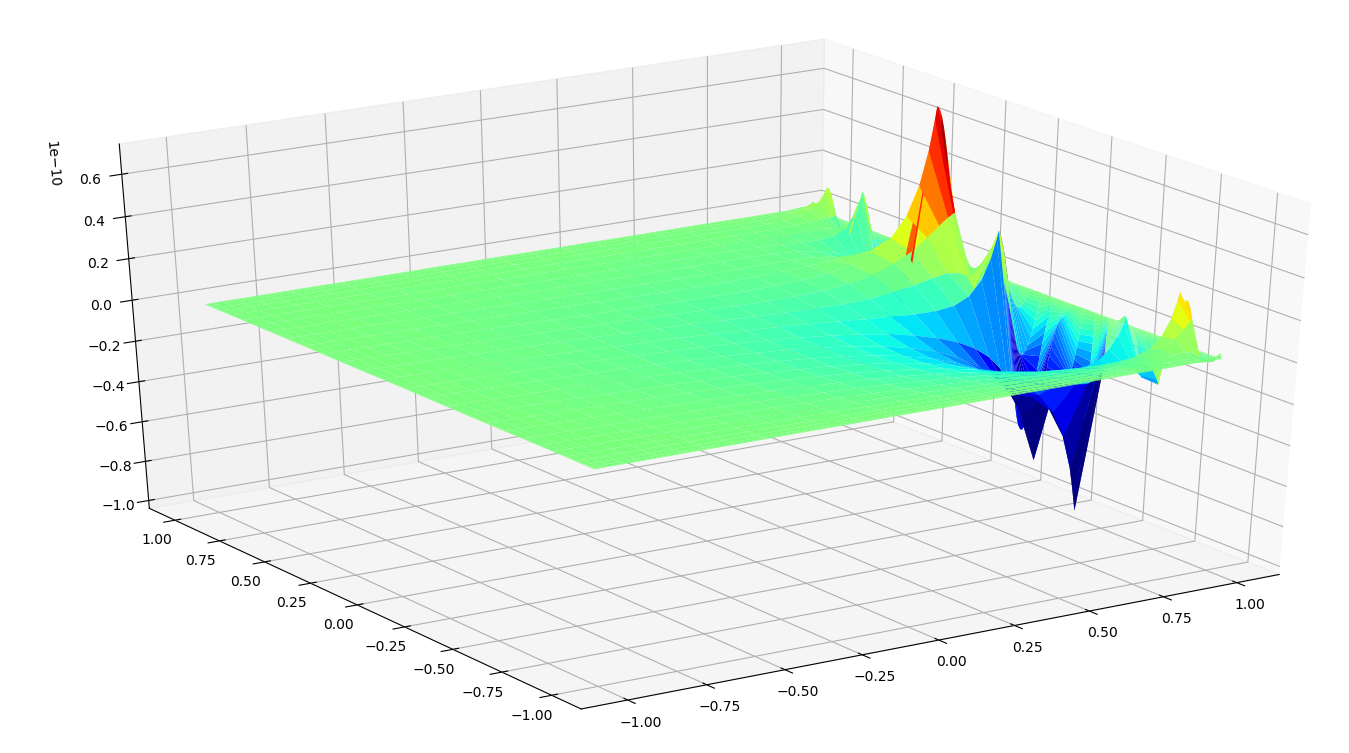} \\
        {\em A solution to Poisson's equation on an element with high aspect ratio attached to an element with low aspect ratio} & {\em Error between the exact solution and the computed solution $(< 10^{-10})$}
    \end{tabular}
    \caption{We can solve Poisson's equation equally accurately on a mesh with and without skinny elements.  Each mesh is the union of the square $[-1, 1] \times [-1, 1]$ and the quadrilateral with vertices $(1, 1), (1, -1), (-0.8, \epsilon), (0.4, 2\epsilon)$.  The top graph shows a solution with $\epsilon = 0.5$ and the bottom shows a solution with $\epsilon = 10^{-6}$.  In both cases, the error is on the order of $10^{-11}$.}
    \label{fig:PoissonSkinnyQuad}
\end{figure}

\begin{figure}
\label{fig:ManyQuadError}
\centering
\includegraphics[width = 0.75\textwidth]{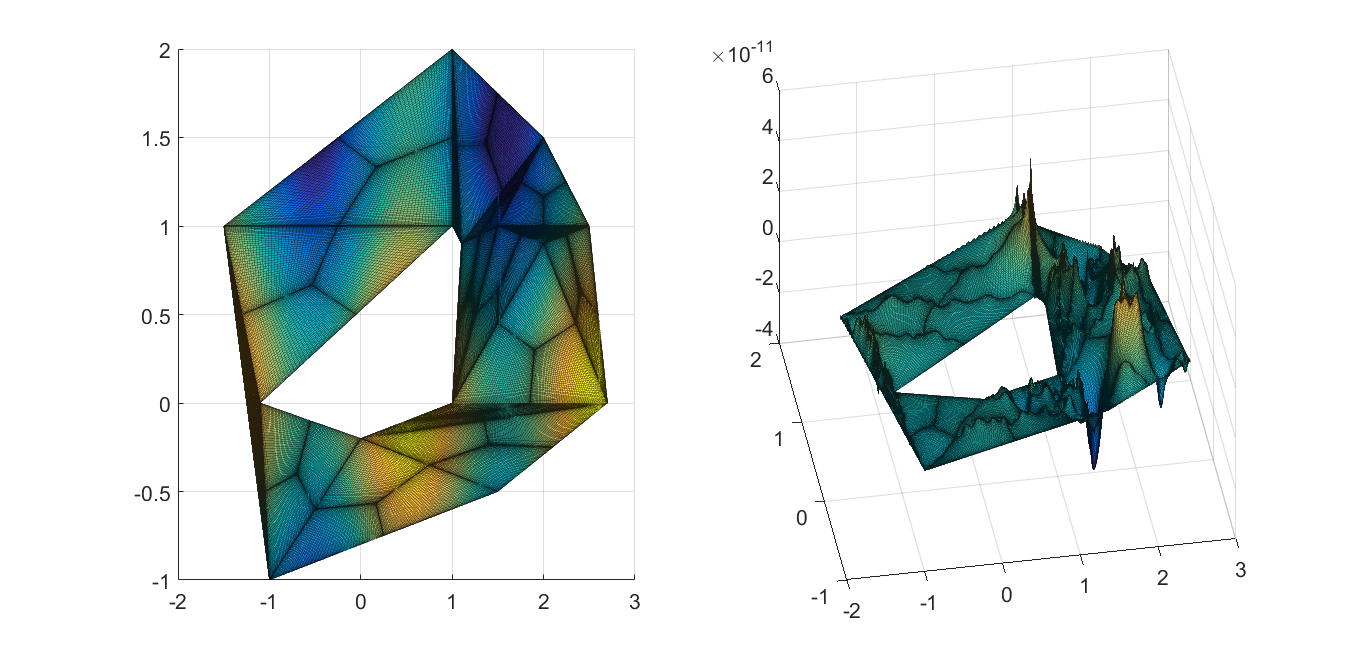}
\caption{Using the Schur complement method, quadrilateral domains can be stitched together to solve differential equations on meshes. Our method is numerically stable even on a mesh containing both skinny and fat triangles.}
\end{figure}

\section{Spectral methods}\label{sec:Ultraspherical}
In order to develop a useful spectral method, one must choose the proper set of basis functions.
Although wavelets and complex exponentials are tempting possibilities, certain sets of polynomials
promise simplicity and easy computability. Although polynomial interpolation is
numerically unstable using equispaced points, it is numerically stable using Chebyshev
points, with the $n$th set of Chebyshev points $x_j = \cos \parens{\frac{n-1-j}{(n-1)\pi}}$
~\cite{trefethen}.

Chebyshev polynomials are an efficient and numerically stable basis for
polynomial approximation~\cite{trefethen}.  We define the $n$th degree Chebyshev
polynomial as $T_n(x) = \cos(n\arccos(x))$, for $x \in [-1, 1]$~\cite[(18.3.1)]{NIST:DLMF}.

\subsection{The ultraspherical spectral method}
Spectral methods are especially useful for solving linear differential equations.
The goal of using spectral methods is to approximate an infinitely-dimensional
linear differential operator with a finite-dimensional sparse matrix operator.
We can represent a function as an infinite vector of Chebyshev coefficients,
and we can truncate that vector at an appropriate size to achieve a sufficiently
accurate approximation of a smooth function.
If we attempted to create a differential operator $D_\lambda$ that mapped a
function in the Chebyshev basis to its $\lambda^{th}$ derivative in the Chebyshev basis,
then that operator would be dense, as most of its matrix elements would be nonzero.
Current spectral methods map Chebyshev polynomials to Chebyshev polynomials,
giving dense, ill-conditioned differentiation matrices.
In contrast, our method constructs $D_\lambda$ to map a function in Chebyshev coefficients
to its $\lambda$th derivative in the ultraspherical basis of parameter $\lambda$.
For every positive real $\lambda$, ultraspherical polynomials of parameter $\lambda$
are orthogonal on $[-1, 1]$ with respect to the weight function $(1 - x^2)^{\lambda - \frac{1}{2}}$
\cite[Eq.~18.3.1]{NIST:DLMF}.  Here, we will ultraspherical polynomials with positive integer parameter.

Using ultraspherical coefficients of parameter $\lambda$, the matrix $\lambda$th
differentiation matrix $D_\lambda$ is
\begin{equation}\label{dmatrix}
D_\lambda = 2^{\lambda-1}(\lambda-1)!
\begin{pmatrix}
\overbrace{0 \quad \cdots \quad 0}^{\lambda \text{ times}} & \lambda\\
& & \lambda + 1\\
& & & \lambda + 2\\
& & & & \ddots
\end{pmatrix}.
\end{equation}
These operators are diagonal, but combining different-order
derivative coefficients leads to a problem:  different-order $D_\lambda$ transform functions
in the Chebyshev basis to ultraspherical bases of different parameters.
The solution is to create basis conversion operators, denoted as $S_\lambda$,
which convert between ultraspherical bases of different orders.
Specifically, $S_0$ converts functions from the Chebyshev basis to the
ultraspherical basis of parameter 1, and $S_\lambda$ converts functions from the
ultraspherical basis of parameter $\lambda$ to the ultraspherical basis of parameter $\lambda + 1$.
For example, we can represent the differential operator $\partialf{u}{x} + 2x$ as
the matrix operator $(D_1 + 2S_0)u$.  The ultraspherical conversion operators take the form
\begin{equation}\label{cmatrix}
S_0 = \frac{1}{2}
\begin{pmatrix}
2 & 0 & -1\\
& 1 & 0 & -1\\
& & 1 & 0 & \ddots\\
& & & 1 & \ddots\\
& & & & \ddots
\end{pmatrix}
, \quad
S_\lambda =
\begin{pmatrix}
1 & 0 & -\frac{\lambda}{\lambda + 2}\\
& \frac{\lambda}{\lambda + 1} & 0 & -\frac{\lambda}{\lambda + 3}\\
& & \frac{\lambda}{\lambda + 2} & 0 & \ddots\\
& & & \frac{\lambda}{\lambda + 3} & \ddots\\
& & & & \ddots
\end{pmatrix}
, \quad \lambda > 0.
\end{equation}
Since $D_\lambda$ and $S_\lambda$ are sparse for all valid $\lambda$,
the final differential operator will be sparse and banded.
Let $L$ be some $n$-order differential operator such that
$L = c_0 + c_1D_1 + c_2D_2 + \ldots + c_nD_n$ where $c_0, \ldots, c_n$ are constants
and $D_n$ is equivalent to $\frac{d^n}{dx^n}$.
If we wish to solve the differential equation $Lu = f$, we create the
matrix operator $\mathbf{L}$.  We also create vectors $\vec{u}$ and $\vec{f}$
as vectors of Chebyshev coefficients approximating $u(x)$ and $f(x)$.
Now, we can write
$\mathbf L = c_0S_{n-1}\cdots S_0 + c_1S_{n-1}\cdots S_1D_1 + \cdots + c_{n-1}D_{n-1} + c_nD_n$,
and we can write the matrix equation $\mathbf L\vec{u} = S_{n-1}\cdots S_0 \vec{f}$.

This method is not limited to equations with constant coefficients.
It is possible to define multiplication matrices of the form $M_\lambda[f]$, where $M_\lambda[f]$ will multiply a vector in the ultraspherical basis of parameter $\lambda$ by function $f$ \cite{townsend_computing_2014}.
The bandwidth of multiplication matrices increases with the degree of $f$, so high speed computation requires $f$ with low polynomial degree.

In order to find a particular solution to the differential equation $Lu = f$, we need boundary conditions.
An $n$th order differential equation requires $n$ distinct boundary conditions.
Since the highest ordered coefficients of $f$ are often very close to zero,
we can simply replace the last $n$ rows of $\mathbf L$ and the last $n$ elements of
$\vec{f}$ with boundary condition rows.
To set $u(x) = a$, we create the boundary condition row of $[T_0(x), T_1(x), T_2(x), \ldots, T_{n-1}(x)]$,
and the corresponding value of $\vec{f}$ is set to $a$.
If we use the boundary condition $u'(x) = a$, we create the boundary condition row
of $[T'_0(x), T'_1(x), T'_2(x), \ldots, T'_{n-1}(x)]$, and the corresponding value of $\vec{f}$ is set to $a$.
We now have a sparse, well-conditioned differential operator,
assuming that the problem is well-posed~\cite{townsend_automatic_2015}.

\subsection{Discrete differential operators in two dimensions}
In two dimensions, we can write any polynomial $p$ as its bivariate Chebyshev expansion
$$\displaystyle p(x, y) = \sum_{i, j = 0}^n a_{ij} T_i(y)T_j(x), \qquad x, y \in [-1, 1].$$
To form a matrix operator, we can stack the columns of the coefficient matrix to form the vector
\begin{equation}
u = [a_{00} , \ a_{10} , \ a_{20} , \ldots , a_{n-1,0} , \ a_{01} , \ a_{11} , \ldots , a_{n-1,1} , \ldots , a_{n-1,n-1}]^T.
\end{equation}
We use Kronecker products (denoted as `$\otimes$') in order to construct differential operators in two dimensions.
For example $\partialf{u}{x}$ is represented as $D_1 \otimes I$ and $\partialf{u}{y}$ is represented as $I \otimes D_1$.
To put together differential operators, we can use the identity
$(A \otimes B) (C \otimes D) = AC \otimes BD$,
so to represent an operator like $\partialf{^3}{x \partial y^2}$, we use $D_1 \otimes D_2$.
We can also add Kronecker products to construct more complicated operators; for instance,
\begin{equation}
\partialf{^2}{x^2} - \partialf{^2}{x \partial y} + 2\partialf{}{x} \quad \Longleftrightarrow \quad
D_2 \otimes S_1S_0 - S_1D_1 \otimes S_1D_1 + 2S_1D_1 \otimes S_1S_0.
\end{equation}
A derivation of this two-dimensional sparse operator construction can be found in~\cite{townsend_automatic_2015}.

Boundary conditions in two dimensions are slightly more complicated than in one dimension.
Consider solving partial differential equations on a square with resolution $n \times n$.
We now remove rows corresponding to a right-hand side term of degree $n-2$ or $n-1$.
These $4n-4$ rows are replaced with boundary conditions corresponding to the $4n-4$ boundary points.

If we want to set Dirichlet conditions at the point $(a, b)$, we create the boundary row
\begin{equation}
  [T_0(a), T_1(a), \ldots, T_{n-1}(a)] \otimes [T_0(b), T_1(b), \ldots, T_{n-1}(b)].
  \label{u row}
\end{equation}

In order to apply Neumann conditions, $S_0^{-1}D_1u$ will give the first derivative of Chebyshev vector $u$ in Chebyshev coefficients.  To find the value the derivative of $u$ at point $a$, we use $[T_0(a), T_1(a), \ldots, T_{n-1}(a)]S_0^{-1}D_1u$.
Thus, a boundary condition row representing $\frac{\partial u}{\partial x}$ at point $(a, b)$ would be
\begin{equation}
  [T_0(a), T_1(a), \ldots, T_{n-1}(a)]S_0^{-1}D_1 \otimes [T_0(b), T_1(b), \ldots, T_{n-1}(b)]
  \label{u_r row}
\end{equation}
and a boundary condition row representing $\frac{\partial u}{\partial y}$ at point $(a, b)$ would be
\begin{equation}
  [T_0(a), T_1(a), \ldots, T_{n-1}(a)] \otimes [T_0(b), T_1(b), \ldots, T_{n-1}(b)]S_0^{-1}D_1.
  \label{u_s row}
\end{equation}

\section{Spectral methods on convex quadrilaterals}\label{sec:Determinant}
The method of constructing matrices described above works only on square domains. We now adapt existing spectral methods to solve differential equations on quadrilaterals, and eventually, quadrilateral meshes.

We use the following bilinear map from the square $[-1, 1]^2$ to the quadrilateral with vertices $(x_1, y_1)\ldots (x_4, y_4)$ in counterclockwise order.
\begin{equation}
x = a_1 + b_1r + c_1s + d_1rs, \qquad
y = a_2 + b_2r + c_2s + d_2rs,
\label{bilinear}
\end{equation}
where $(x, y)$ refers to coordinates on the quadrilateral and $(r, s)$ refers to coordinates on the square.
\begin{equation}
\begin{aligned}
a_1 &= \frac{1}{4}\parens{x_1 + x_2 + x_3 + x_4},\qquad &b_1 = \frac{1}{4}\parens{x_1 - x_2 - x_3 + x_4},\\
c_1 &= \frac{1}{4}\parens{x_1 + x_2 - x_3 - x_4},       &d_1 = \frac{1}{4}\parens{x_1 - x_2 + x_3 - x_4},
\end{aligned}
\label{coordinates to coefficients}
\end{equation}
and $a_2$, $b_2$, $c_2$, and $d_2$ are similarly defined with $y_1$, $y_2$, $y_3$, and $y_4$.
The transformation from the quadrilateral to the square is more complicated, but it will not be needed.

When the quadrilateral domain is mapped onto the square, the differential equations on the quadrilateral are also distorted.  Through the use of the chain rule, we obtain
\begin{equation}
\begin{aligned}
u_x &= u_r r_x + u_s s_x,\\
u_y &= u_r r_y + u_s s_y,\\
u_{xx} &= u_{rr}(r_x^2) + 2u_{rs}(r_x s_x) + u_{ss}(s_x^2) + u_r r_{xx} + u_s s_{xx},\\
u_{xy} &= u_{rr}(r_x r_y) + u_{rs}(r_x s_y + s_x r_y) + u_{ss}(s_x s_y) + u_rr_{xy} + u_s s_{xy},\\
u_{yy} &= u_{rr}(r_y^2) + 2u_{rs}(r_y s_y) + u_{ss}(s_y^2) + u_r r_{yy} + u_s s_{yy}.
\end{aligned}
\label{derivatives}
\end{equation}
We still need to evaluate derivatives $r_x$, $r_y$, $s_x$, $s_y$, $r_{xx}$, $r_{xy}$, $r_{yy}$, $s_{xx}$, $s_{xy}$, and $s_{yy}$.
Using the Inverse Function Theorem, we can invert the Jacobian of the transformation to obtain
\begin{equation}
\begin{aligned}
\begin{pmatrix}
r_x & r_y\\
s_x & s_y
\end{pmatrix}
=
\begin{pmatrix}
x_r & x_s\\
y_r & y_s
\end{pmatrix}^{-1}
=
\frac{1}{\det(x, y)}
\begin{pmatrix}
y_s & -x_s\\
-y_r & x_r
\end{pmatrix},
\qquad \det(x, y) = x_r y_s - x_s y_r.
\end{aligned}
\label{jacobian}
\end{equation}
Using these equations, we can also find second derivatives $r_{xx}$, $r_{xy}$, $r_{yy}$, $s_{xx}$, $s_{xy}$, and $s_{yy}$.  These derivatives can be represented as a bivariate cubic polynomial divided by $\det(r, s)^3$.

In order to keep our differential operators sparse, we simply multiply the equations and right-hand side by $\det(r, s)^3$.
That way, we only multiply the operator by multiplication matrices with polynomial degree $< 3$, rather than by a dense approximation of a rational function, and our differential operators remain sparse and banded.

\section{Interface conditions}\label{sec:Interface}
In order to actually do something useful with quadrilateral domains, it is necessary to somehow stitch them together into a larger mesh.
Here, we will first consider the case of stitching two quadrilaterals together along an edge.
We found that the most efficient way to stitch together multiple domains was to use the Schur Complement Method \cite{Mathew}.
Essentially, we first solve for the values along the interface between the two domains, and then we use that interface to find the bulk solution inside of each domain.  The advantage to using the Schur Complement Method is that once we have solved for the interface solution, we can compute the solution on each element in parallel.

In order to use the Schur Complement method on two quadrilaterals, we must have a linear system of the form
\begin{equation}
\begin{bmatrix}
A_{11} & 0 & A_{1\Gamma}\\
0 & A_{22} &  A_{2\Gamma}\\
A_{\Gamma 1} & A_{\Gamma 2} & A_{\Gamma \Gamma}
\end{bmatrix}\!\!
\begin{bmatrix}
u_1 \\ u_2 \\ u_\Gamma
\end{bmatrix}
=
\begin{bmatrix}
f_1 \\ f_2 \\ f_\Gamma
\end{bmatrix}.
\end{equation}

Here, $u_1$ is the solution on the first quadrilateral, $u_2$ is the solution on the second quadrilateral, and $u_\Gamma$ is the solution on the interface between $u_1$ and $u_2$.
We will represent $u_\Gamma$ as a vector of values at Chebyshev points.

We also have $A_{11}$ and $A_{22}$ for solving partial differential equations on quadrilateral elements.
Finally, $A_{1\Gamma}$ and $A_{2\Gamma}$ constrain $u_\Gamma$ to $u_1$ and $u_2$, and $A_{\Gamma 1}$ and $A_{\Gamma 2}$ ensure that the solution across the interface is differentiable.  $A_{\Gamma \Gamma}$ is zero, as it is not necessary here.

First, we need to set $A_{1\Gamma}$ and $A_{2\Gamma}$.
As shown in Figure \ref{nodes}, if an edge of $u_1$ is connected to $u_\Gamma$, then we connect all of the nodes on that edge of $u_1$ to $u_\Gamma$ except for the node on the counterclockwise corner of that edge.
This arrangement ensures that all nodes on the exterior boundary of the domain are assigned the proper boundary conditions.
It also avoids redundant boundary conditions when multiple quadrilaterals meet at a point on the interior of the mesh.
\begin{figure}
  \label{nodes}
  \centering
  \includegraphics[width=0.6\textwidth]{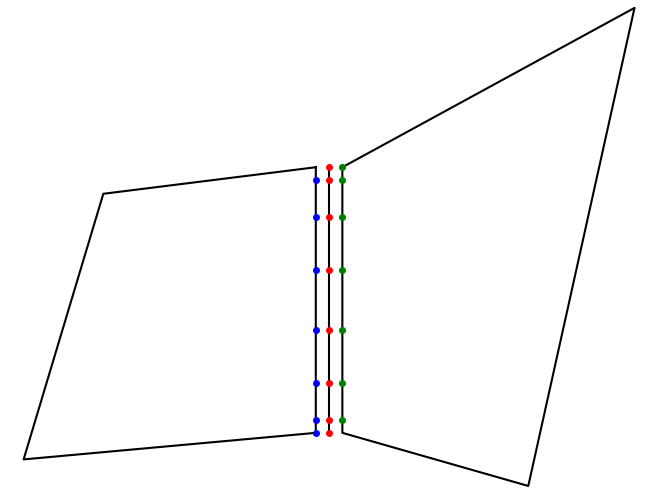}
  \caption{Let the left quadrilateral be $u_1$, the right quadrilateral be $u_2$, and the middle line be $u_\Gamma$.
  Only $n-1$ points on the rightmost edge of $u_1$ are connected to $u_\Gamma$, since the top right corner is excluded.
  Similarly, the bottom left corner is excluded from $u_2$
  While this arrangement may seem unnecessarily complicated on the surface, it facilitates stitching together large numbers of quadrilaterals.}
\end{figure}

Now, we need to ensure that the solution is differentiable across the interface with $A_{\Gamma 1}$ and $A_{\Gamma 2}$.
In order to do so, we need to find the derivative perpendicular to the boundary.
Let us say that the boundary has endpoints $(x_1, y_1)$ and $(x_2, y_2)$.
Let us define
\begin{equation}
  \alpha = \frac{\Delta x}{\sqrt{\Delta x^2 + \Delta y^2}}, \qquad \beta = \frac{\Delta y}{\sqrt{\Delta x^2 + \Delta y^2}}.
\end{equation}

Let $z = \beta x - \alpha y$, so the derivative normal to the boundary is $u_z = \beta u_x - \alpha u_y$. We want this derivative in terms of $r$ and $s$. Using~\eqref{derivatives}, we have
\begin{equation}
  u_x = u_r r_x + u_s s_x, \qquad u_y = u_r r_y + u_s s_y.
\end{equation}
From~\eqref{jacobian}, we have
\begin{equation}
  u_x = \frac{1}{\det(r, s)}\parens{u_r y_s - u_s y_r}, \qquad u_y = \frac{1}{\det(r, s)}\parens{-u_r x_s + u_s x_r}.
\end{equation}
By differentiating~\eqref{bilinear}, we have
\begin{equation}
  \begin{aligned}
    u_x &= \frac{1}{\det(r, s)}\parens{u_r (c_2 + d_2 r) - u_s (b_2 + d_2 s)}\\
    u_y &= \frac{1}{\det(r, s)}\parens{-u_r (c_1 + d_1 r) + u_s (b_1 + d_1 s)}.
  \end{aligned}
\end{equation}
Finally, we can generate boundary rows for $u_r$ and $u_s$ with~\eqref{u_r row} and~\eqref{u_s row}.
Each row of $A_{\Gamma 1}$ represents the derivative of $u_1$ perpendicular to the interface at one of the $n$ Chebyshev points on the interface,
and each row of $A_{\Gamma 2}$ represents the negative derivative of $u_2$ at each Chebyshev point on the boundary.

Now, we want to ensure continuity across the boundary at the two endpoints. We can simply replace the differentiability conditions corresponding to the endpoints of the interface with continuity conditions from~\eqref{u row}.

\section{Constructing differential operators on meshes}\label{sec:ElementMethod}
The key to designing a mesh solver is strict bookkeeping.
Every quadrilateral in the mesh is given a unique number, as is every vertex and every edge.
We can define our mesh with a list of coordinates of vertices, and a list of the vertices contained in each triangle.
It is important that all of the vertices are numbered in counterclockwise order.
Next, each edge must be given a unique global number.
Edges also have local definitions---an edge can be expressed as two vertices, or as
a quadrilateral number and a local edge number on the quadrilateral.
We generate arrays to convert between these global and local definitions.
Finally, we must construct an array that keeps track of whether vertices are on the
inside or boundary of the domain.  Each vertex is also assigned a single edge
adjacent to that vertex.  This ``vertex list'' is essential for constructing non-singular
boundary conditions.

Next, we define $u_\Gamma$ such that elements $n(k-1)+1 \!:\! nk$ of $u_\Gamma$ correspond to the $k$th interior edge. Thus, columns $n(k-1)+1 \!:\! nk$ of $A_{i\Gamma}$ and rows $n(k-1)+1 \!:\! nk$ of $A{\Gamma i}$ correspond to interface conditions across the $k$th interior edge.

We now generate a matrix $A_{i}$ for each quadrilateral.
If quadrilateral $i$ and $j$ share an edge with interior edge number $k$, then columns $n(k-1)+1 \!:\! nk$ of $A_{i\Gamma}$
and columns $n(k-1)+1 \!:\! nk$ of $A_{j\Gamma}$ are set like in the previous section.
Next, rows $n(k-1)+1 \!:\! nk$ of $A{\Gamma i}$ and rows $n(k-1)+1 \!:\! nk$ of $A{\Gamma j}$ are set to ensure differentiability, with the endpoints set to continuity conditions.
Finally, we check the two endpoints of the edge on the vertex list.
If an endpoint is listed as an interior vertex and the vertex list marks edge $k$, then the continuity conditions are removed for that endpoint and replaced with differentiability conditions across the edge at that endpoint.
That way, we make sure that the solution is fully continuous and almost fully differentiable without singularities in the linear system.

\section{Optimization}\label{sec:Optimization}
When solving systems of equations, banded matrices are typically efficient to solve.
Our matrices for finding solutions on quadrilaterals are ``almost banded,'' meaning that most of the elements lie close to the main diagonal, but a few rows extend the whole length of the matrix.
These dense rows correspond to boundary conditions, as defined in Section \ref{sec:Ultraspherical}.
Therefore, sparse LU decomposition and Gaussian elimination have large backfill---they are forced to set many zero matrix elements to nonzero values, resulting in a much more computationally intensive solve.
We found that we can use the Woodbury matrix identity to efficiently replace the rows lying outside of the bandwidth,
creating a banded matrix that can easily be stored as an LU decomposition.
The Woodbury matrix identity can compute the inverse a matrix given the inverse of another matrix and a rank-$k$ correction.  It satisfies
\begin{equation}
(A + UCV)^{-1} = A^{-1} - A^{-1}U(C^{-1} + VA^{-1}U)^{-1}VA^{-1},
\end{equation}
where $A$ is $n$-by-$n$, $U$ is $n$-by-$k$, $V$ is $k$-by-$n$, and $C$ is $k$-by-$k$~\cite{higham}.

In our case, we regard our matrix as $A + UCV$, where $A$ is a matrix that is fast to solve and $UCV$ is a low--rank correction.
We generate $V$ by stacking every row of $A + UCV$ that corresponds to a boundary condition, since these rows are the only dense rows.
We define $U$ as a sparse binary matrix with exactly one `1' in each column and no more than one `1' in each row, and $C$ is simply the identity matrix.  The replacement boundary condition rows are 1's in columns corresponding to the lowest-order coefficients of the solution.  Our matrix $A$ is now sparse and banded, so fast sparse LU decomposition is possible.  Therefore, once the LU decomposition is completed, back-substitution can be evaluated for each individual right-hand side extremely rapidly, allowing a system of variables with over 60,000 degrees of freedom to be solved in under a second.

Another place for optimization is the matrix $$\Sigma = A_{\Gamma \Gamma}
 - \sum_j A_{\Gamma j} A_{jj}^{-1} A_{j \Gamma}.$$
The matrix $\Sigma$ is typically quite large, so it is important to be efficiently solvable.
Consider $A_{\Gamma j} A_{jj}^{-1} A_{j \Gamma}$.
This matrix has nonzero columns only in columns of $A_{j \Gamma}$
that contain nonzero values, and it has nonzero rows only in rows of $A_{\Gamma j}$
that contain nonzero values.
Fortunately, most rows of $A_{\Gamma j}$ and columns of $A_{j \Gamma}$ are zero.
In order to optimize the bandwidth of $\Sigma$, we must optimize the placement of the nonzero elements of $A_{\Gamma j}$ and $A_{j \Gamma}$.

In general, if there are $m$ different boundaries between quadrilaterals, and $n$ is the discretization size,
then $u_\Gamma$ is composed of $m$ blocks of $n$ elements each, with each block
corresponding to a boundary between two quadrilaterals.  Each boundary is assigned a number, such that
elements $n(a-1)+1 \! :\!  na$ of $u_\Gamma$ correspond to boundary $a$.
If boundary $a$ is contained by quadrilateral $j$, then $A_{\Gamma j}$ has nonzero rows
$n(a-1)+1 \! : \! na$ and $A_{j \Gamma}$ has nonzero columns $n(a-1)+1 \! : \! na$.
Thus, if two quadrilaterals contain boundaries $a$ and $b$, then $\Sigma$ has a bandwidth of
at least $(|a-b| + 1)n$.  Therefore, in order to minimize the bandwidth of $\Sigma$, we
minimize the maximum $|a - b|$ for any two boundaries contained by the same quadrilateral.
The theoretical minimum is $\OO(\sqrt N)$, where $N$ is the number of elements, and
standard techniques can efficiently get pretty close to this minimum.  Therefore,
computing the LU decomposition of $\Sigma$ has a time complexity of $\OO(N^2n^3)$
and a space requirement of $\OO(N^{1.5}n^2)$.  The time complexity of back-substitution
is $\OO(N^{1.5}n^2)$.

Using the Woodbury identity for all of the matrices $A_{jj}$ and computing $A_{jj}^{-1} A_{j \Gamma}$
has a time complexity of $\OO(N n^4)$.  Back substitution for all $A_{jj}$ has a time complexity of $\OO(N n^3)$, and the storage space for the matrices is also $\OO(N n^3)$.
In order to solve partial differential equations on a mesh, there is an $\OO(N^2 n^3 + N n^4)$ precomputation time, and for each right hand side, the solve time is $\OO(N^{1.5} n^2 + N n^3)$.
It may be possible to use an iterative method to multiply a vector by $\Sigma^{-1}$, reducing the factors of $\OO(N^2)$
and $\OO(N^{1.5})$ to $\OO(N)$.

\section{The Navier--Stokes equations}\label{sec:NavierStokes}
\begin{figure}
\label{vortices}
\centering
\begin{tabular}{c c}
\includegraphics[width = 0.45\textwidth]{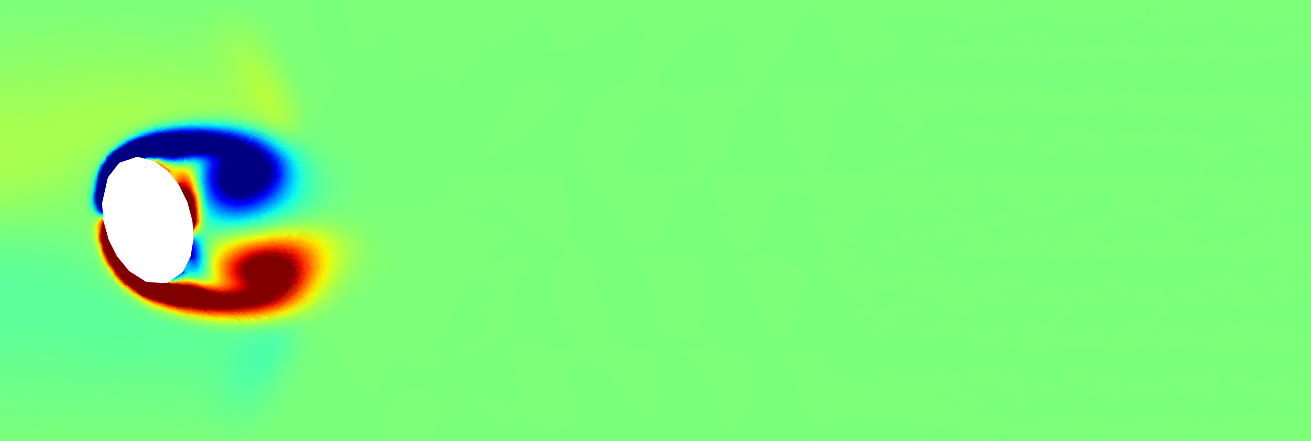} &
\includegraphics[width = 0.45\textwidth]{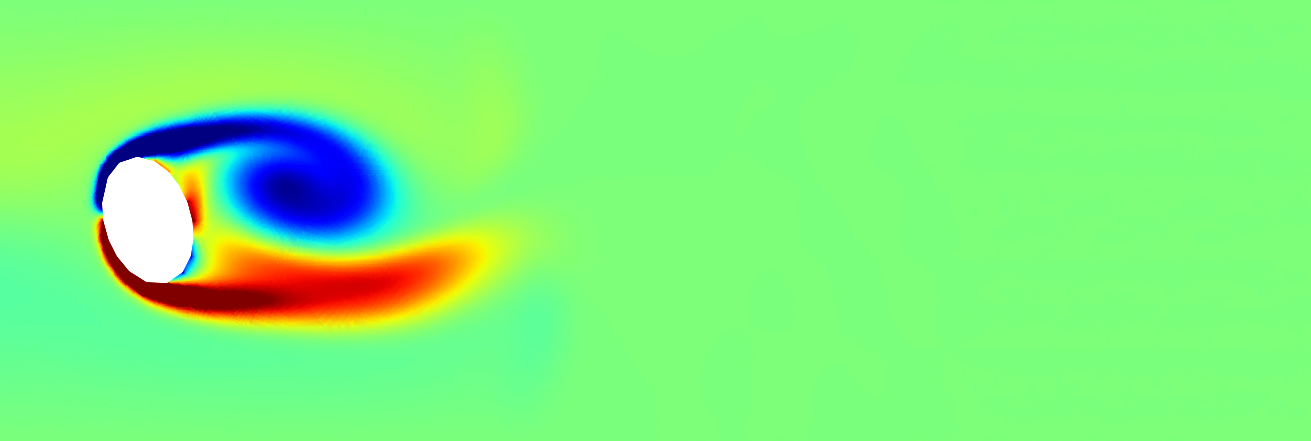} \\
Frame 500 ($t = 0.00833$~s) & Frame 1000 ($t = 0.01667$~s)\\ \\
\includegraphics[width = 0.45\textwidth]{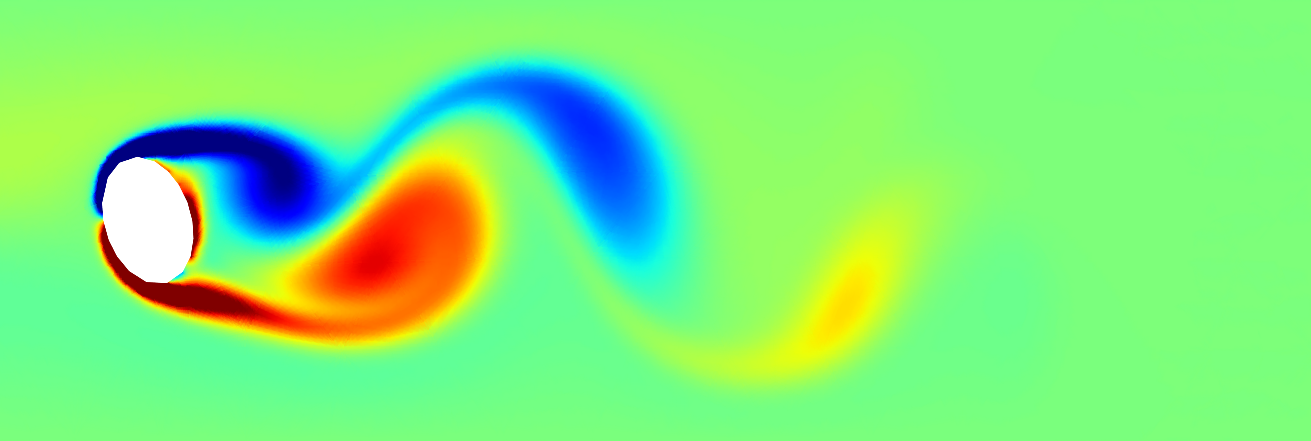} &
\includegraphics[width = 0.45\textwidth]{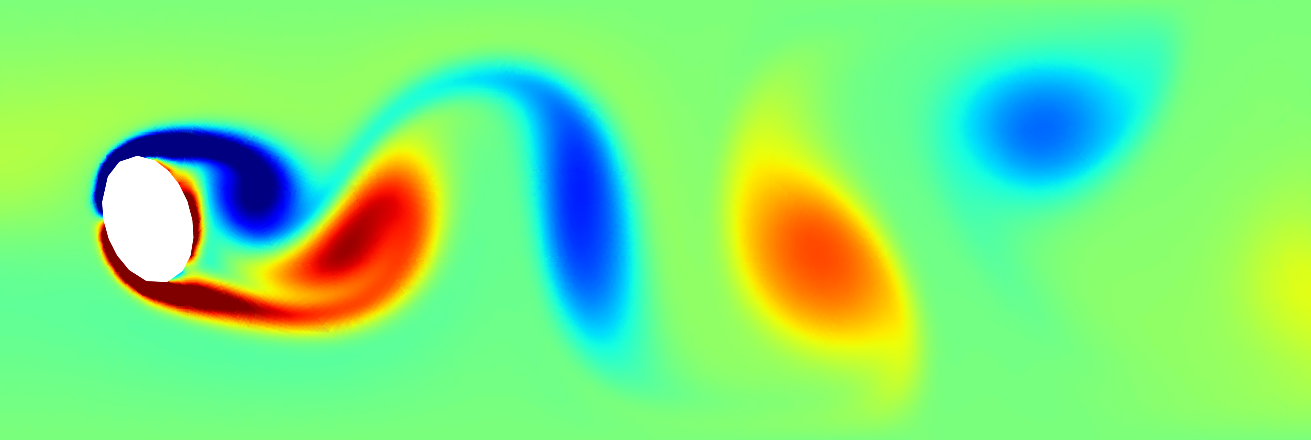} \\
Frame 2000 ($t = 0.03333$~s) & Frame 3000 ($t = 0.0500$~s)\\ \\
\includegraphics[width = 0.45\textwidth]{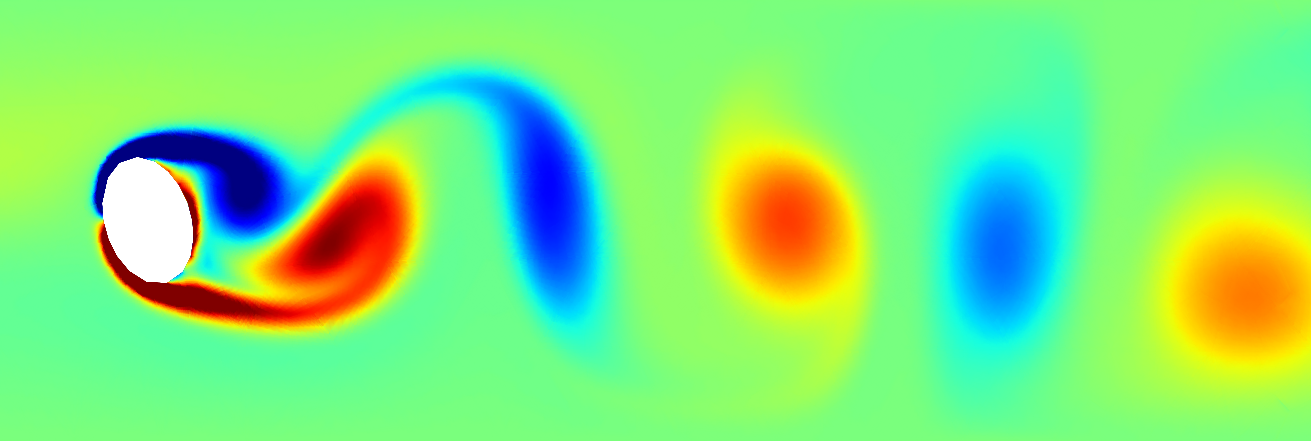} &
\includegraphics[width = 0.45\textwidth]{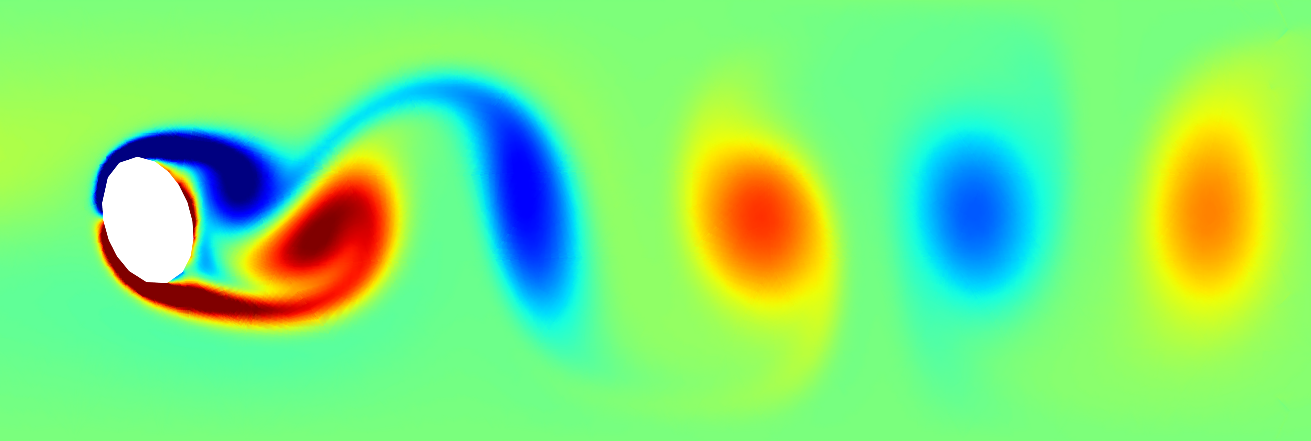} \\
Frame 4000 ($t = 0.06667$~s) & Frame 5000 ($t = 0.08333$~s)\\ \\
\includegraphics[width = 0.45\textwidth]{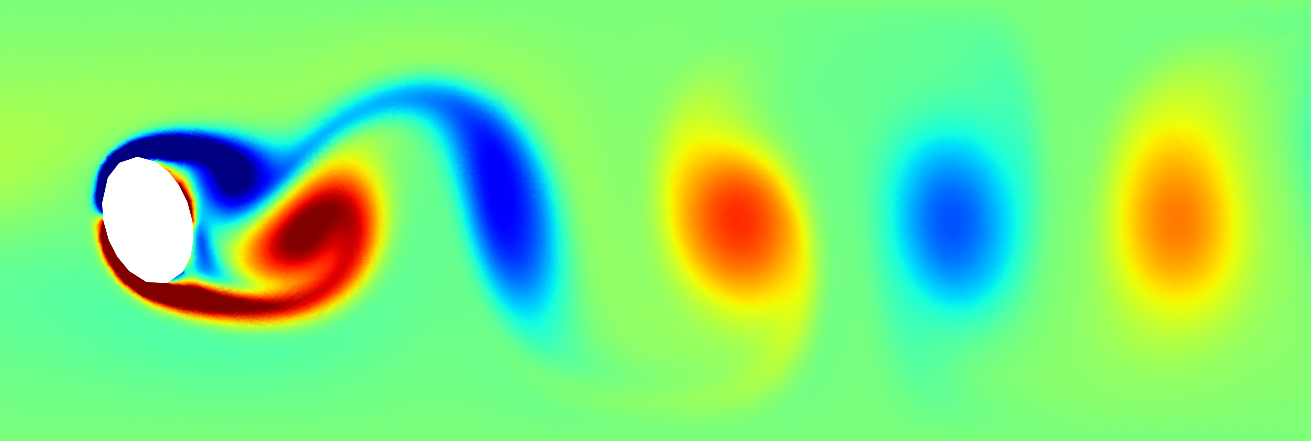} &
\includegraphics[width = 0.45\textwidth]{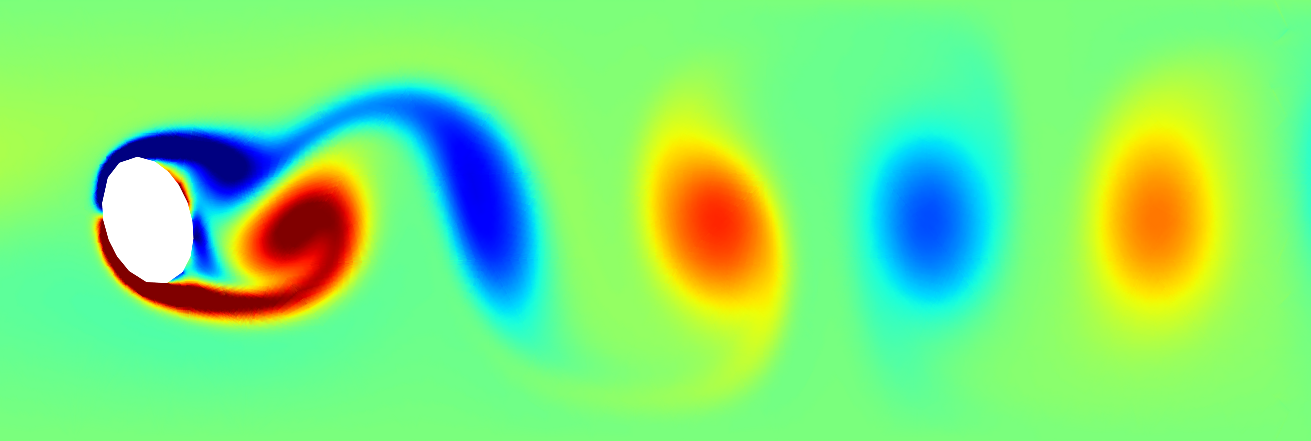} \\
Frame 6000 ($t = 0.10000$~s) & Frame 7000 ($t = 0.11667$~s)
\end{tabular}
\caption{A vortex street develops around a tilted elliptical object in fluid flow.
Red regions have positive vorticity (counterclockwise) and blue regions have negative vorticity.
The simulation is equivalent to air flow through a 1~cm~$\times$~3~cm rectangle at $0.6$~m/s.
Each time step is $1.667 \times 10^{-5}$~s.}
\end{figure}

We decided to demonstrate our method with the direct numerical simulation of the Navier--Stokes equations in two dimensions at moderately high Reynolds numbers.
For incompressible flows, the Navier--Stokes equation are
\begin{equation}
\partialf{\mathbf u}{t} + (\mathbf u \cdot \nabla) \mathbf u + \nabla p = \nabla^2 u, \qquad \nabla \cdot \mathbf u = 0,
\end{equation}
where $\mathbf u$ is the velocity vector field and $p$ is the internal pressure field.
Note that these equations set density and viscosity to 1, so the Reynolds number can be varied solely by changing the velocity of the fluid or the dimensions of the domain.
Since these equations are coupled and nonlinear, we use the multistep method described in \cite{nseqs}.  We can use a first-order projection method to linearize and decouple the Navier--Stokes equations.  With no-slip boundary conditions, we have
\begin{equation}
\begin{aligned}
&\nabla^2 \mathbf u^{n+1/2} - \frac{\mathbf u^{n+1/2}}{\triangle t} = (\mathbf u^n \cdot \nabla) \mathbf u^n - \frac{\mathbf u^n}{\triangle t}, &\qquad &\mathbf u^{n+1/2} = 0, \quad \text{on } \partial \Omega\\
&\nabla^2 p^{n+1} = \frac{\mathbf u^{n+1/2}}{\triangle t}, &\qquad &\partialf{p^{n+1}}{\mathbf n} = 0, \quad \text{on } \partial \Omega\\
&\mathbf u^{n+1} = \mathbf u^{n+1/2} - \triangle t \nabla p^{n+1}, & &
\end{aligned}
\end{equation}
where $\mathbf n$ is the normal vector to the boundary~\cite{nseqs}.

Navier--Stokes simulations are extremely useful in the form of wind tunnel simulations, where a test object is placed in a steady flow and analyzed.
For an incompressible flow, different boundary conditions are applied at the inlet, outlet, walls, and test object.
Typically, the test object will have no-slip boundary conditions, and the walls will have free-slip boundary conditions.  The inlet has a constant fluid velocity and zero pressure gradient, and the outlet has a constant pressure and zero velocity gradient.

Figure~\ref{vortices} shows a wind tunnel simulation of an object in a moving fluid.  A tilted ellipse was chosen to generate asymmetries in the flow, leading to vortex shedding earlier in the simulation.  

\section*{Acknowledgments}\label{sec:Acknowledgment}
We thank Pavel Etingof, Slava Gerovitch, and Tanya Khovanova for their work on the MIT PRIMES program, which allowed us to collaborate together.  In March 2017, the first author was awarded the second prize at the Regeneron Science Talent search worth \$175,000 based on this work that will cover the tuition fees at MIT starting in August 2017. We are sincerely humbled by Regeneron's very generous support. We thank Grady Wright for giving us advice on the implementation of the Navier--Stokes simulation.  We also thank Dan Fortunato and Heather Wilber.  This work was supported by National Science Foundation grant No.~1645445.

\appendix
\section{Determinants}\label{dets}
Let $(x_1, y_1)$, $(x_2, y_2)$, and $(x_3, y_3)$ be the vertices of an arbitrary triangle in counterclockwise order.
By the shoelace formula, the area of the triangle is
\begin{equation}
A = \frac{1}{2}\Big[\big(x_1y_2 + x_2y_3 + x_3y_1\big) - \big(x_2y_1 + x_3y_2 + x_1y_3\big)\Big].
\label{shoelace}
\end{equation}
We can construct three quadrilaterals from the triangle by partitioning the triangle along the line segments connecting the centroid of the triangle to the midpoints of the sides.  Without loss of generality, let our quadrilateral contain vertex 1 of the triangle, so it has the following vertices in counterclockwise order:
\begin{equation}
\big(x_1, y_1\big), \quad \bigg(\frac{x_1 + x_2}{2}, \frac{y_1 + y_2}{2}\bigg), \quad \bigg(\frac{x_1 + x_2 + x_3}{3}, \frac{y_1 + y_2 + y_3}{3}\bigg), \quad \bigg(\frac{x_1 + x_3}{2}, \frac{y_1 + y_3}{2}\bigg).
\label{quadrilateral coordinates}
\end{equation}
Using equations \eqref{coordinates to coefficients} and \eqref{quadrilateral coordinates}, we have
\begin{equation}
x = \frac{1}{24}\Big[\big(14 x_1 + 5 x_2 + 5 x_3\big) + \big(4 x_1 - 5 x_2 + x_3\big)r + \big(4 x_1 + x_2 - 5 x_3\big)s + \big(2 x_1 - x_2 - x_3\big)rs \Big],
\end{equation}
and similarly for $y$.
We can use equations \eqref{bilinear} and \eqref{jacobian} to find that
\begin{equation}
\det(r, s) = \big(b_1 c_2 - b_2 c_1\big) + \big(b_1 d_2 - b_2 d_1\big)r + \big(c_2 d_1 - c_1 d_2\big) s.
\end{equation}
Plugging in values for $b_1$, $c_1$, $d_1$, $b_2$, $c_2$, and $d_2$, we have
\begin{equation}
\begin{split}
\det(r, s) = \frac{1}{576}\Big[&\big( (4 x_1 - 5 x_2 + x_3) (4 y_1 + y_2 - 5 y_3) - (4 y_1 - 5 y_2 + y_3) (4 x_1 + x_2 - 5 x_3) \big)\\
 + &\big((4 x_1 - 5 x_2 + x_3) (2 y_1 - y_2 - y_3) - (4 y_1 - 5 y_2 + y_3) (2 x_1 - x_2 - x_3) \big)r\\
 + &\big((4 y_1 + y_2 - 5 y_3) (2 x_1 - x_2 - x_3) - (4 x_1 + x_2 - 5 x_3) (2 y_1 - y_2 - y_3) \big)s \Big].
\end{split}
\label{big ugly}
\end{equation}
We can note that
\begin{equation}
\Bigg(\sum_{i=1}^n v_i x_i \Bigg)\Bigg(\sum_{i=1}^n w_i y_i \Bigg) - \Bigg(\sum_{i=1}^n w_i x_i \Bigg) \Bigg(\sum_{i=1}^n v_i y_i \Bigg) = \sum_{1 \leq i < j \leq n} (x_i y_j - x_j y_i)(v_i w_j - v_j w_i),
\label{summation}
\end{equation}
and use this identity to simplify \eqref{big ugly} to obtain
\begin{equation}
\det(r, s) = \frac{1}{96}\Big[\big(x_1y_2 + x_2y_3 + x_3y_1\big) - \big(x_2y_1 + x_3y_2 + x_1y_3\big)\Big]\big(4 + r + s\big).
\end{equation}
Substituting in \eqref{shoelace} and letting $A$ be the area of the triangle gives the desired result of
\begin{equation}
\det(r, s) = \frac{A\big(4 + r + s\big)}{48}.
\end{equation}

\end{document}